\documentclass[reqno]{amsart}
\usepackage[english]{babel}
\usepackage{amscd,amssymb,amsmath,amsfonts,latexsym,amsthm}
\usepackage{ulem}
\usepackage{inputenc}
\usepackage{graphicx,color}
\newtheorem{thm}{Theorem}[section]
\newtheorem{cor}[thm]{Corollary}
\newtheorem{lem}[thm]{Lemma}
\newtheorem{prop}[thm]{Proposition}
\newtheorem{defn}[thm]{Definition}
\newtheorem{exam}[thm]{Example}
\newtheorem{rem}[thm]{Remark}

\numberwithin{equation}{section}

\newcommand{\ZZ}{\mathbb{Z}}

\def\cal{\mathcal }

\begin{document}

\title[Maximal solvable extensions]{On maximal solvable extensions of nilpotent Lie algebras}

\author{B.A. Omirov, G.O. Solijanova}

\address{Bakhrom A. Omirov \newline \indent
Institute for Advanced Study in Mathematics,
Harbin Institute of Technology, Harbin 150001 \newline \indent
Suzhou Research Institute, Harbin Institute of Technology, Harbin  215104, Suzhou, China}
\email{{\tt omirovb@mail.ru}}

\address{Gulkhayo O. Solijanova \newline \indent
National University of Uzbekistan, 100174, Tashkent, Uzbekistan,
\newline \indent
V.I. Romanovskiy Institute of Mathematics, Uzbekistan Academy of Sciences, Uzbekistan}\email{{\tt gulhayo.solijonova@mail.ru}}

\begin{abstract} In this paper, we provide a complete description of complex maximal solvable extensions for a certain class of nilpotent Lie algebras. In particular, we show that, up to isomorphism, a solvable extension of a $d$-locally diagonalizable nilpotent Lie algebra is unique and is realized as the semidirect product of its nilradical with a maximal torus. This result resolves a conjecture of \v{S}nobl concerning the uniqueness of maximal solvable extensions under the condition $d$-locally diagonalizability on the nilradical. Moreover, we extend this description to the setting of Lie superalgebras and present an alternative method for constructing such maximal solvable extensions. Finally, we discuss further aspects and open questions related to maximal solvable extensions of nilpotent Lie algebras
\end{abstract}



\maketitle

\section{Introduction}

The theory of Lie algebras has been the subject of extensive research in modern algebra and theoretical physics, yielding many beautiful results and generalizations. In the theory of finite-dimensional Lie algebras, Levi's decomposition, together with fundamental results of Malcev's and Mubarakzjanov's, reduces the description problem of Lie algebras to the study of nilpotent ones along with their special kind of derivations \cite{Mal}, \cite{mubor} and representations of semisimple algebras.

In fact, the description of finite-dimensional nilpotent Lie algebras is an immense problem, that is why they are usually studied under some additional restrictions. In contrast, the structure of non-nilpotent solvable Lie algebras is relatively well understood. The point here is that these algebras can be reconstructed from their nilradical using non-nilpotent derivations. Mubarakzjanov's method of constructing solvable Lie algebras in terms of their nilradicals is effectively used \cite{mubor}. The investigation of solvable Lie algebras with some special types of nilradicals arises from different problems in physics and has been the subject of various papers \cite{AnCaGa1,AnCaGa2,BoPaPo,Cam,NdWi,SnKa,SnWi,TrWi,WaLiDe} and references therein. For instance, solvable Lie algebras whose nilradicals are quasi-filiform, $\mathbb{N}$-graded, Abelian and triangular were considered in \cite{AnCaGa2,Cam,NdWi,TrWi}.

One particularly noteworthy motivation for the classification and identification of Lie algebras, as discussed in \cite{snoble}, arises when two systems of differential equations of the same order describing different physical processes have isomorphic Lie algebras of infinitesimal point symmetries that, by construction, are realized by vector fields.

Recall that a Lie algebra is called {\it characteristically nilpotent} if all its derivations are nilpotent. For a survey on characteristically nilpotent Lie algebras, we refer the reader to \cite{Ancochea111}. It should be noted that characteristically nilpotent Lie algebras can not be identified as the nilradical of non-nilpotent solvable Lie algebras. In other words, the nilradical of a non-nilpotent solvable Lie algebra is non-characteristically nilpotent.

Based on numerous results on the descriptions of solvable Lie algebras with given nilradicals, which demonstrate that the description of non-maximal extensions of nilpotent Lie algebras is boundless problem, while maximal extensions are unique (up to isomorphism), \v{S}nobl \cite{Snobl} hypothesized the following

{\bf Conjecture:} Let $\mathfrak{n}$ be a complex nilpotent Lie algebra that is not characteristically nilpotent. Let $\mathfrak{s}, \tilde{\mathfrak{s}}$ be solvable Lie algebras with the nilradical $\mathfrak{n}$ of maximal dimension in the sense that no such solvable algebra of larger dimension exists. Then, $\mathfrak{s}$ and $\tilde{\mathfrak{s}}$ are isomorphic.

It is worth noting that in \cite{mubor}, two non-isomorphic maximal solvable extensions of the three-dimensional Heisenberg algebra over the real numbers were constructed. This example indicates that, to ensure the uniqueness of maximal solvable extensions, our study need to be carried out over an algebraically closed field of characteristic zero.

In \cite{Gorbatsevich}, Gorbatsevich discovered a counterexample to the conjecture; he constructed two non-isomorphic maximal solvable extensions of the nilpotent Lie algebra, which are the direct sum of characteristically nilpotent and abelian ideals. In order to find out the condition under which the conjecture is false, following Gorbatsevich's example, we construct an example (Example \ref{exam15}) that gives an idea what condition should be imposed on a nilpotent Lie algebra to guarantee a positive solution of \v{S}nobl's conjecture.

Recall that Lie superalgebras extend the classical Lie algebras by introducing a $\mathbb{Z}_2$-grading, this grading scheme distinguishes between elements of different parity, allowing for the introduction of a graded Lie bracket, which extends the classical Lie bracket to accommodate both commuting and anticommuting elements. The structure theory of Lie superalgebras is more complex than the theory of Lie algebras due to the presence of odd elements, nevertheless significant progress has been achieved in the classification of simple and basic classical Lie superalgebras. The results on the algebraic structures of Lie superalgebras, as well as the physics outcomes within the framework of supersymmetry, can be found in \cite{Cooper, Kac, Lusztig, Molev, Serganova} and references therein.

Note that Leibniz algebras represent another generalization of Lie algebras, which inherit a property of that the right multiplication operators on an element of an algebra are derivations \cite{Loday}. They
are naturally generalize of Lie algebras such that many classical result of Lie algebras are true for Leibniz algebras \cite{Ayupov}, \cite{Feldvoss}.

The main goal of this paper is to investigate the uniqueness of maximal solvable extensions of nilpotent Lie algebras and to describe the properties of their derivations. We also examine whether these results can be extended to generalizations of Lie algebras such as Lie superalgebras and Leibniz algebras.

This paper studies the structure of maximal solvable extensions of nilpotent Lie algebras. The central concept introduced is that of a $d$-locally diagonalizable nilpotent Lie algebra, defined by three conditions ensuring that nil-independent derivations generate a maximal torus and act diagonally on root subspaces, with no zero root present. The innerness of all derivations of maximal solvable extensions of $d$-locally diagonalizable Lie algebras is established. We also examine whether the description of maximal solvable extensions can be extended to generalizations of Lie algebras such as Lie superalgebras and Leibniz algebras. Some conjectures are proposed regarding the broader validity of the uniqueness results. This results gives the positive answer to \v{S}nobl's Conjecture for the case of nilradical is complex $d$-locally diagonalizable nilpotent Lie algebra.

The paper is organized as follows. In Section 3, we introduce the notion of $d$-locally diagonalizability and provide a complete description of maximal solvable extensions of $d$-locally diagonalizable nilpotent Lie algebras. In particular, by applying suitable basis transformations induced by inner automorphisms of the nilradical, we establish that a maximal solvable extension of a $d$-locally diagonalizable nilpotent Lie algebra is isomorphic to the semidirect product of its nilradical with a maximal torus (Theorem~\ref{mainthm}) and the uniqueness of this extension follows from the conjugacy of maximal tori in nilpotent Lie algebras.

In Section 4, assuming the validity of an analogue of Lie's theorem, we establish the uniqueness, up to isomorphism, of maximal solvable extensions for $d$-locally diagonalizable nilpotent Lie superalgebras (Theorem~\ref{mainthm2}). Furthermore, we provide an explicit method for constructing maximal tori in nilpotent Lie (super)algebras, which is then used to realize these maximal extensions. We also provide sufficient conditions under which non-maximal solvable extensions admit outer derivations (Proposition~\ref{thmH1neq0}) and establish that maximal solvable extensions of $d$-locally diagonalizable nilpotent Lie algebras possess only inner derivations (Theorem~\ref{thm5.4}). Finally, we present an example demonstrating that, in contrast to the Lie algebra case, the uniqueness of maximal solvable extensions does not generally hold for $d$-locally diagonalizable nilpotent Leibniz algebras. Motivated by examples concerning maximal solvable extensions, we propose two natural conjectures and outline directions for further investigation into other types of maximal solvable extensions.

Throughout the paper we consider finite-dimensional vector spaces and algebras over the field $\mathbb{C}$. In addition, unless otherwise noted we shall assume non-nilpotent solvable algebras and nilpotent algebras which are non-characteristically nilpotent.

\section{Preliminaries}

In this section, we present the definitions and preliminary results that will be used in this paper.  For basic notions and results on nilpotent and solvable Lie algebras we refer the reader to \cite{Jac}. We introduce only the notion of a torus on a nilpotent Lie algebra, along with related results, which play an important role in our study.

\begin{defn} A torus on a Lie algebra $\mathcal{L}$ is an abelian subalgebra of $Der(\mathcal{L})$ consisting of semisimple endomorphisms. A torus is said to be maximal if it is not strictly contained  in any other torus. We denote a maximal torus by $\cal T$.
\end{defn}

In the case of algebraically closed field semisimple endomorphism means diagonalizable. Choosing an appropriate basis of a Lie algebra we can bring the mutually commuting diagonalizable derivations into a diagonal form, simultaneously. Applying simple arguments on maximal fully reducible subalgebras of the linear Lie algebra $\operatorname{Der}(\mathcal{N})$ and Mostow's results (see Theorem 4.1 in \cite{Mostow}), one can conclude that any two maximal tori of the nilpotent Lie algebra $\mathcal{N}$ are conjugate under an inner automorphism $\varphi\in \operatorname{Aut}(\mathcal{N})$.
The automorphism $\varphi$ is a similarity in view of the identity:
$\widetilde{\cal T}=\varphi \circ \cal T \circ \varphi^{-1}.$
Consequently, two arbitrary maximal tori of a nilpotent Lie algebra have the same dimension, which is called {\it rank} of the nilpotent Lie algebra  (denoted by $\operatorname{rank}(\cal N)$).

Consider root spaces decomposition of a Lie algebra $\cal N$ with respect to its a maximal torus $\cal T$:
$$\cal N=\cal N_{\alpha_1}\oplus \cal N_{\alpha_2} \oplus \dots \oplus \cal N_{\alpha_n},$$ where
$\cal L_{\alpha_i}=\left\{n_{\alpha_i}\in \cal L \ |  \ [n_{\alpha_i}, x]=\alpha_i(x)n_{\alpha_i},\, \ \forall \ x \in \cal T\right\}$ and a root $\alpha_i$ belongs to the dual space of $\cal T.$  Then $\cal T$ admits a basis $\{t_1, \dots, t_s\}$ which satisfies the condition
\begin{equation}\label{eq3.1.1}
t_i(n_{\alpha_j})=\alpha_{i,j}n_{\alpha_j}, 1\leq i,j\leq s,
\end{equation}
where $\alpha_{i,i}\neq 0, \ \alpha_{i,j}=0$ for $i\neq j$ and $n_{\alpha_j}\in \cal L_{\alpha_j}.$

For a nilpotent Lie algebra such that $\operatorname{dim}(\cal N /\cal N^2)=k$ we denote by $W=\left\{\alpha_i\in {\cal T}^\ast \ : \ \cal N_{\alpha_i}\neq 0\right\}$ the roots system of  $\cal N$  associated to  $\cal T,$
by $\Psi=\{\alpha_1, \dots, \alpha_s, \alpha_{s+1}, \dots, \alpha_k\}$ the set of simple roots of $W$ such that any root subspace with index in $\Psi$ contains at least one generator element of $\cal L$ and by $\Psi_1=\left\{\alpha_1,\ldots,\alpha_s\right\}$ the set of primitive roots such that any non-primitive root can be expressed by a linear combination of them. In fact, any root $\alpha\in W$ we have $\alpha=\sum\limits_{\alpha_i\in \Psi_1}p_i\alpha_i$ with $p_i\in \mathbb{Z}.$

Let $\mathcal{R}=\cal N \oplus \cal Q$ be a solvable Lie algebra with nilradical $\mathcal{N}$ and let $\mathcal{Q}$ be a complementary subspace of $\mathcal{N}$. From \cite{mubor}, it is known that the operator $\operatorname{ad}_{x|\mathcal{N}}$ is non-nilpotent for every element $x \in \mathcal{Q}$. This result motivates the definition of nil-independent derivations.
\begin{defn} The derivations $d_1, d_2, \dots, d_s$ of a Lie algebra $\mathcal{N}$ are said to be nil-independent if $\sum\limits_{i=1}^{s}\alpha_i d_i$ is a non-nilpotent derivation
for any non-zero scalars $\alpha_1, \alpha_2, \dots, \alpha_s \in \mathbb{C}$.
\end{defn}
Then in terms of nil-independent derivations, we obtain an upper bound for the dimension of $\mathcal{Q}$. Specifically, the dimension of $\mathcal{Q}$ is bounded above by the maximal number of nil-independent derivations of the nilradical $\mathcal{N}$.

In what follows, we introduce the central concept underlying the present paper.
\begin{defn} A solvable Lie algebra $\mathcal{R}$ with nilradical $\cal N$ is called a  maximal solvable extension of the nilpotent Lie algebra $\mathcal{N}$, if $\operatorname{dim} \cal Q$ is maximal. In other words, $\cal R$ is a solvable Lie algebra with the nilradical $\cal N$ of maximal dimension in the sense that no such solvable algebra of larger dimension exists.
\end{defn}

The next proposition, gives a necessary condition for a Lie algebra with only inner derivations to have a non-trivial center.
\begin{prop}\cite{Leger1} \label{prop1}  Let $\mathcal{L}$ be a Lie algebra over a field of characteristic $0$ such that $Der(\mathcal{L})=\operatorname{InDer}(\mathcal{L}).$ If the center of the Lie algebra $\mathcal{L}$ is non-trivial, then $\mathcal{L}$ is not solvable and the radical of $\mathcal{L}$ is nilpotent.
\end{prop}

Below we present the definition of a Lie superalgebra and related concepts.
\begin{defn} A  Lie superalgebra is a $\ZZ_2$-graded vector space  $\mathcal{L}=\mathcal{L}_{\bar 0} \oplus \mathcal{L}_{\bar 1}$ equipped with a bilinear bracket $[ \ , \ ]$, which is agreed with $\ZZ_2$-gradation and for arbitrary homogeneous elements $x, y, z$ satisfies the conditions

\begin{enumerate}
\item[1.] $[x,y]=-(-1)^{|x| |y|}[y,x],$

\item[2.]   $(-1)^{|x| |z|}[x,[y,z]]+(-1)^{|x| |y|}[y,[z,x]]+(-1)^{|y||z|}[z,[x,y]]=0$ {\it ( Jacobi superidentity).}
\end{enumerate}
\end{defn}

A superderivation $d$ of degree $|d|, |d|\in \ZZ_2,$ of a Lie superalgebra $\mathcal{L}$ is an endomorphism $d \in End(\mathcal{L})_{|d|}$ with the property
\begin{align}\label{der}
d([x,y])=[d(x),y] + (-1)^{|x| \cdot |d|}[x, d(y)].
\end{align}

If we denote by $Der(\mathcal{L})_{s}$ the space of all superderivations of degree $s$, then $Der (\mathcal{L})=Der(\mathcal{L})_0\oplus Der(\mathcal{L})_1$ forms the Lie superalgebra of superderivations of $\mathcal{L}$ with respect to the bracket $[d_i,d_j]=d_i\circ d_j - (-1)^{|d_i| |d_j|} d_j\circ d_i$, where $Der(\mathcal{L})_0$ consists of even superderivations (or just derivation) and $Der(\mathcal{L})_1$ of odd superderivations. For an element $x\in \mathcal{L}$ the operator $ad_x : \mathcal{L} \to \mathcal{L}$ defined by $ad_x(y)=[y,x]$ is a superderivation of the Lie superalgebra $\mathcal{L}$, which is called {\it inner superderivation}. The set of all inner superderivations is denoted by $\operatorname{InDer}(\mathcal{L})$ and it forms an ideal of the Lie superalgebra $Der(\mathcal{L})$.

Note that Engel's theorem remains to be valid for Lie superalgebras too (see \cite{scheunert}). But, nevertheless, neither analogue of Lie's theorem nor its corollaries for solvable Lie superalgebras are true, in general. However, it is known that the analogue of Lie's theorem for a solvable Lie superalgebra  $\mathcal{R}=\mathcal{R}_0\oplus \mathcal{R}_1$ is true if and only if $[\mathcal{R}_1,\mathcal{R}_1] \subseteq [\mathcal{R}_0,\mathcal{R}_0]$ (see Proposition 5.2.4 in \cite{Kac}). In addition, on nilpotency of the square of a solvable Lie superalgebra is not always true.

In the paper \cite{Wang2002}, the definition of torus of a Lie superalgebra is given in a similar way as for a Lie algebra. Since (in the case of algebraically closed field) the diagonalizability of odd superderivation leads to its nullity, one can assume that a torus on a complex Lie superalgebra $\mathcal{L}$ is an abelian subalgebra of the Lie algebra $Der(\mathcal{L})_0,$ consisting of diagonalizable endomorphisms.

\section{Maximal solvable extensions of some nilpotent Lie algebras}\label{sec3}

In this section, we provide the sufficient conditions under which maximal solvable extensions are unique.

Thanks the equality $[\operatorname{ad}_x,\operatorname{ad}_y]=\operatorname{ad}_{[x,y]}$ for all elements $x, y\in \mathcal{R}$, the solvability of $\mathcal{R}$ is equivalent to the solvability of $\operatorname{InDer}(\mathcal{R})$. Applying Lie's theorem to the algebra $\operatorname{InDer}(\mathcal{R})$, we conclude that there exists a basis of $\mathcal{R}$ such that all operators $\operatorname{ad}_{x}, x\in \mathcal{R},$ in the basis have upper-triangular matrix forms. In particular, all operators $\operatorname{ad}_{x|\mathcal{N}}, x \in \mathcal{Q}$, also admit upper-triangular matrix forms in the chosen basis of $\mathcal{N}$. Since $\mathcal{Q}$ is a vector space, the matrix entries of these operators are linear functionals on $\mathcal{Q}$.

The analysis of the existing results and the examples presented below provides insight into the restrictions that must be imposed on a nilpotent Lie algebra to ensure the uniqueness of its maximal solvable extension.

\begin{exam} \label{exam15} Let $\cal N_9=\operatorname{Span}\{e_1, \dots, e_9\}$ be nilpotent Lie algebra with multiplications table
$$\left\{\begin{array}{lllllll}
[e_{1}, e_{2}] = e_{3},&[e_{1}, e_{3}] = e_{4},&[e_{1}, e_{4}] = e_{5},&[e_{1}, e_{6}] = e_{7},&[e_{1}, e_{8}] = e_{9},\\[1mm]
[e_{2}, e_{3}] = e_{8},&[e_{2}, e_{4}] = e_{9},&[e_{2}, e_{5}] = e_{9},&[e_{4}, e_{3}] = e_{9}.
\end{array}\right.$$

Straightforward computations lead that all derivations of the algebra have upper-triangular matrix forms and $\cal T=Span\{t_{\alpha}, t_{\beta}\}$ with $t_{\alpha}=diag(0, 1, 1, 1, 1, 0, 0, 2, 2)$, $t_{\beta}=diag(0,0,0,0,0,1, 1, 0, 0)$ generates a maximal torus on $\cal N_9.$ One can check that solvable Lie algebras $\cal R_1=\cal N\rtimes T$ and $\cal R_2=\cal N\rtimes \widetilde{T}$ with $\widetilde{\cal T}=Span\{t_{\alpha}+d, t_{\beta}\},$ where $d\in Der(\cal N_9)$ defined as
$$d(e_2)=2 e_{3}-e_{4}, \ d(e_{3}) =2 e_{4}-e_{5}, \ d(e_4)=2 e_{5}, \ d(e_i)=0, \ i\neq 2, 3, 4$$
are non-isomorphic.

Note that alge`bra $\cal N_9$ satisfies the following properties: (i) diagonals of upper-triangular derivations generates a maximal tours of $\cal N_9$; (ii) the restriction and projection of $t_{\alpha}+d$ on $\cal N_{\alpha}$ is non-diagonal; (iii) $0 \in W$.
\end{exam}

Based on the structures of maximal solvable extensions illustrated in Gorbatsevich's example and Example \ref{exam15}, we introduce the notion of a $d$-locally diagonalizable nilpotent Lie algebra.

For a nilpotent Lie algebra $\cal N$, we fix a basis in which all nil-independent derivations are simultaneously represented by upper-triangular matrices. Further, we shall work exclusively with this basis; that is, all nil-independent derivations will be assumed a priori to have upper-triangular matrix representations.

\begin{defn} \label{defn3.2}  A nilpotent Lie algebra $\cal N$ is called $d$-locally diagonalizable, if it is satisfies three conditions:
\begin{itemize}
\item[(i)] the diagonals of all nil-independent derivations generate a maximal torus $\cal T$ on $\cal N$;
\item[(ii)] the restrictions and projections of all nil-independent derivations on each root subspaces with respect to the maximal torus $\cal T$ are diagonal.
\item[(iii)] $0 \notin W$.
\end{itemize}
\end{defn}
Taking into account the conjugacy of any two maximal tori, we conclude that the definition of $d$-locally diagonalizability is independent of the choice of torus and hence, the concept is well-defined.

Note that $d$-locally diagonalizable algebra is non-characteristically nilpotent algebra.

The following example shows that condition (ii) of Definition \ref{defn3.2} is not a consequence of condition (iii).
\begin{exam}\label{exam3.3} Consider the tensor product of the Lie algebra $\cal N_1$ and the associative-commutative algebra $\cal N_2$ defined by
$$\begin{array}{cccc}\cal N_1:\left\{\begin{array}{lllll}
 [x_1,x_i]=x_{i+1},&2\le i\le 7,&[x_2,x_3]=x_7,\\[2mm]
 [x_2,x_4]=x_8,&[x_2,x_5]=x_8,&[x_4,x_3]=x_8,
\end{array}\right.&\cal N_2:\left\{\begin{array}{lll}[y_1,y_2]=y_3,\\[2mm]
[y_2,y_1]=y_3.\end{array}\right.\end{array}$$

Let $\cal N_3$ denote the quotient algebra $(\cal N_1\otimes \cal N_2)/\cal J$, where the ideal $\cal J$ is generated by $\operatorname{Span}\{x_1\otimes y_3,\ x_2\otimes y_3,\ x_8\otimes y_1,\ x_8\otimes y_2\}$. Next, we construct the semidirect product $\cal N=\cal N_3\ltimes \cal N_4$, where $\cal N_4=\operatorname{Span}\{z_1, \dots, z_8\}$, whose products determined by the relations $[z_2,z_1]=z_{6}, \ [z_3,z_2]=z_7$ and the action of $\cal N_3$ on $\cal N_4$ is specified through the following products:
$$\left\{\begin{array}{lllllll}
[z_1,x_2\otimes y_1]=z_5,&[z_1,x_3\otimes y_1]=z_5,& [z_2,x_1\otimes y_1]=z_6,&[z_2,x_1\otimes y_2]=z_6,&[z_2,x_3\otimes y_1]=z_6,\\[2mm]
[x_2\otimes y_2,z_3]=z_7,&[z_3,x_3\otimes y_2]=z_7,&[z_{4},x_4\otimes y_2]=z_8,&[z_{4},x_7\otimes y_1]=z_8.\end{array}\right.$$

It can be shown that all derivations of $\cal N$ can be represented by upper-triangular matrices. Furthermore, conditions (i) and (iii) are satisfied, while condition (ii) fails to hold.
\end{exam}

For a given nilpotent Lie algebra $\mathcal{N}$ and its torus $\mathcal{T}$, we denote by $\cal R_{\mathcal{T}}$ the solvable Lie algebra $\mathcal{N}\rtimes \mathcal{T}$ with products $[\cal N, \cal T]=\cal T(\cal N).$

\begin{rem} \label{rem3.4}
\

\begin{itemize}
\item It should be noted that for any two maximal tori $\cal T$ and $\widetilde{\cal T}$ of $\cal N$, the corresponding solvable Lie algebras $\cal R_{\cal T}$ and $\cal R_{\widetilde{\cal T}}$ are isomorphic via an isomorphism $\psi$ satisfying
$$\psi_{|\cal N}:=\varphi, \ \psi(\cal T):=\varphi\circ \cal T \circ \varphi^{-1},$$
where $\varphi$ is an inner automorphism of $\cal N,$ conjugating the maximal tori.

\item For a maximal solvable extension of a nilpotent Lie algebra $\mathcal{N}$ satisfying condition~(i) of Definition~\ref{defn3.2}, one has
$\operatorname{rank}(\mathcal{N}) = \operatorname{codim} \mathcal{N}.$ In other words, in condition~(i) one may assume a maximal torus instead of an arbitrary torus. Indeed, let $\mathcal{R}=\cal N \oplus \cal Q$ and let $\{x_1, \dots, x_s\}$ be a basis of $\cal Q$. The existence of solvable Lie algebra $\cal R_{\cal T}$ implies that $s \geq \operatorname{rank}(\cal N).$ On the other hand, by the maximality of $\cal R$ and the condition (i), the diagonal parts of $\operatorname{ad}_{{x_i}|\mathcal{N}}, \ 1\leq i \leq s,$ generate a torus of $\cal N$. Hence, $s\leq  \operatorname{rank}(\cal N)$ and we get $\operatorname{rank}(\cal N)=\operatorname{dim}(\cal Q)$.
\end{itemize}
\end{rem}

The following example shows that for a maximal solvable extension of a non
$d$-locally diagonalizable nilpotent Lie algebra, the equality
$\operatorname{rank}(\cal N)=\operatorname{codim} \cal N$ does not hold, in general.
\begin{exam} \label{exam3.5} One can show that in an appropriate basis of the algebra $\cal N_3$ from Example \ref{exam3.3}, all derivations have upper-triangular matrix forms. Hence, the following derivations constitute a maximal set of nil-independent derivations of $\cal N_3$:
$$d_1=\sum\limits_{i\in \cal I_{1}}E_{i,i}-\sum\limits_{i\in \cal J_{1}}E_{i,i}-E_{17,18}-E_{7,9}-E_{8,10},\quad d_2=\sum\limits_{i\in \cal I_3}E_{i,i}+E_{7,9}+E_{8,10}+E_{17,18},\ $$
$$d_3=\sum\limits_{i\in \cal I_{2}}E_{i,i}-\sum\limits_{i\in \cal J_{2}}E_{i,i},\quad  d_4=\sum\limits_{i\in \cal I_4}E_{i,i},\quad   d_5=\sum\limits_{i\in\cal I_5}E_{i,i},$$

where
$$\cal  I_1=\{3,4,11,12,15,19\},\ \ \cal I_{2}=\{1,11,12,13,14\},\ \ \cal J_1=\{9,10,18\}, \ \ \cal J_{2}=\{4,6,8,10,12,14,19,20\},$$
$$\cal I_3=\cal J_1\cup\{5,6,11,12,13,14,16,19,20\},\ \ \cal I_4=\cal J_1\cup\{7,8,13,14,17,20\},\ \ \cal I_5=
\cal J_2\cup \{2,15,16,17,18\}.$$


The maximal solvable extension of $\cal N_3$ with complementary subspace $\cal Q=\{z_1,z_2,z_3,z_4,z_5\}$, acting on $\cal N_3$ via ${\operatorname{ad}_{z_i}}_{|\cal N}=d_i, \ 1\le i\le 5,$ and satisfying $[\cal Q, \cal Q]=0$, has $\operatorname{codim}\cal N_3=5$. Meanwhile,  $\operatorname{Span}\{d_1+d_2, d_3,d_4,d_5\}$ forms a four-dimensional maximal torus. Consequently, $\operatorname{codim} \cal N_3 > \operatorname{rank}(\cal N)$.

Note that conditions (i) and (ii) of Definition \ref{defn3.2} do not hold, whereas condition (iii) holds true.
\end{exam}

Now we are going to establish that a maximal solvable extension of a $d$-locally diagonalizable nilpotent Lie algebra is unique up to isomorphism.

Let $\mathcal{N}$ be a $d$-locally diagonalizable nilpotent Lie algebra and let $\mathcal{R} = \mathcal{N} \oplus \mathcal{Q}$ be a maximal solvable extension of $\mathcal{N}$. We fix a basis of $\mathcal{N}$ in which all nil-independent derivations are represented by upper-triangular matrices. Setting
$\mathcal{Q} = \operatorname{Span}\{x_1, \dots, x_s\}$, we obtain a root space decomposition $\mathcal{N} = \mathcal{N}_{\alpha_1} \oplus \mathcal{N}_{\alpha_2} \oplus \cdots \oplus \mathcal{N}_{\alpha_n}$
with respect to a maximal torus $\mathcal{T}$, which is generated by the diagonal parts of $\operatorname{ad}_{x_i|\mathcal{N}}$, $1 \le i \le s$. Therefore, one may assume that the following products hold:
\begin{equation}\label{eq11}
\begin{array}{lllll}
[n_{\alpha_{i}},x_i]&=&\alpha_{i,i}n_{\alpha_{i}}+
\sum\limits_{t=i+1}^{s}m_{i,i}^{t}+
\sum\limits_{\alpha\in W\setminus \Psi_1}(*),& 1\leq i \leq s;\\[1mm]
[n_{\alpha_{i}},x_j]&=&
\sum\limits_{t=i+1}^sm_{i,j}^{t}+
\sum\limits_{\alpha\in W\setminus \Psi_1}(*),& 1\leq i\neq j \leq s,\\[1mm]
\end{array}
\end{equation}
where $n_{\alpha_{i}}\in \cal N_{\alpha_i}$,  $m_{*,*}^{t}\in \cal N_{\alpha_t}$ and $(*)$ denote the terms of $\cal N_{\alpha}$ with $\alpha\in W\setminus \Psi_1.$

For convenience, we consider the product of inner automorphisms of the following forms:
\begin{equation}\label{eq3.4.1}
n_{\alpha_i}'= exp(d_{i,r})\circ\dots\circ exp(d_{i,1})(n_{\alpha_i}),
\end{equation}
where
$$\begin{array}{lllll}
d_{i,p}=A^{-1}(ad_{ x_l|\cal V_{p-1}}-t_l) \ \mbox{with operator} \ t_l \ \mbox{as in} \ \eqref{eq3.1.1} \ \mbox{and} \\[3mm]
\cal V_0= \cal N, \quad \cal V_{p-1}=exp(d_{i,p-1})(\cal N), \quad  2\le p\le r.\\[3mm]
\end{array}$$

\begin{lem} \label{lem1.8} There exists a basis of $\cal R$ such that the following products hold true:

$$\begin{array}{lllll}
i)& [n_{\alpha_{i}},x_i]&=&\alpha_{i,i}n_{\alpha_{i}}+\sum\limits_{\alpha\in W\setminus \Psi_1}(*), &1\leq i \leq s;\\[1mm]
ii)& [n_{\alpha_{i}},x_j]&=&\sum\limits_{\alpha\in W\setminus \Psi_1}(*),& 1\leq i\neq j\leq s;\\[1mm]
iii) &[x_i,x_j]&=&\sum\limits_{\alpha\in W\setminus \Psi_1}(*),&1\leq i, j \leq s.
\end{array}$$
\end{lem}

\begin{proof} Part {\it i}) is obtained by applying the change of \eqref{eq3.4.1} with $A=\alpha_{i,i}, \ l=i$ and $r=s-i.$

Part {\it ii}). Since $[\cal R, \cal R]\subseteq \cal N$, we have $[\cal Q, \cal Q]\subseteq \cal N$ and hence, we conclude that $[n_{\alpha_{i}},[x_i,x_j]]=\sum\limits_{\alpha\in W\setminus \Psi_1}(*).$ Therefore, applying {\it i}) and the products of \eqref{eq11} in the following chain of equalities with $1\leq i\neq j\leq s:$
$$[n_{\alpha_{i}},[x_i,x_j]]=
[[n_{\alpha_{i}},x_i],x_j]-[[n_{\alpha_{i}},x_j],x_i]=\alpha_{i,i}
[n_{\alpha_{i}}, x_j]-[[n_{\alpha_{i}},x_j],x_i]+\sum\limits_{\alpha\in W\setminus \Psi_1}(*)$$
we deduce $[n_{\alpha_{i}}, x_j]=\sum\limits_{\alpha\in W\setminus \Psi_1}(*).$

Part {\it iii}). Applying {\it ii}) in the equality
$[x_k, [x_i,x_j]]=[[x_k, x_i],x_j] - [[x_k, x_j],x_i]$
with pairwise non equal indexes $k, i, j$ from the set $\{1, \dots, s\},$ we derive

$$[x_i,x_j]=\theta_{i,j}^i+\theta_{i,j}^j+\sum\limits_{\alpha\in W\setminus \Psi_1}(*)$$ with $\theta_{i,j}^i\in \cal N_{\alpha_i}, \theta_{i,j}^j\in \cal N_{\alpha_j}.$

Now setting $x_i'=x_i-\sum\limits_{t=1, t\neq i}^{s}\theta_{i,t}^t ,\  1\leq i\leq s,$
we obtain $[x_i',x_j']=\sum\limits_{\alpha\in W\setminus \Psi_1}(*)e_{\alpha}$.
\end{proof}

Consider now the obtained products modulo $\cal N^2$ in $\cal R$

\begin{equation}\label{eq12}
\left\{\begin{array}{llllll}
[n_{\alpha_{i}},x_i]&\equiv& \alpha_{i,i}n_{\alpha_{i}}+\sum\limits_{t=s+1}^{k}m_{i,i}^t, & 1\leq i\leq s,\\[1mm]
[n_{\alpha_{i}},x_j]&\equiv& \sum\limits_{t=s+1}^{k}m_{i,j}^t, & 1\leq i\neq j \leq s,\\[1mm]
[n_{\alpha_{i}},x_j]&\equiv& \alpha_{i,j}n_{\alpha_i}+\sum\limits_{t=i+1}^{k}m_{i,j}^t, & s+1\leq i\leq k, \ 1\leq j \leq s,\\[1mm]
[x_i,x_j]&\equiv &\sum\limits_{t=s+1}^{k}v_{i,j}^t, & 1\leq i, j \leq s,\\[1mm]
\end{array}\right.
\end{equation}
where $\alpha_{i,j}=\alpha_{i}(x_j).$

Let us introduce denotation

$$\prod\limits^h(i,t)=\prod\limits_{q=1}^{h}
\delta_{\alpha_{i,q},\alpha_{t,q}}, \quad \mbox{with} \quad
\delta_{\alpha_{i,q},\alpha_{t,q}}=
\left\{\begin{array}{llll}
1, \quad if \quad \alpha_{i,q}=\alpha_{t,q}\\[1mm]
0, \quad otherwise
\end{array}.\right.$$

The product property leads to the equality $\prod\limits^h(i,t)\prod\limits^h(p,q)=\prod\limits^h(i,t)(p,q).$

In the next lemma we get the exact expressions for the products $[n_{\alpha_{i}},x_j]$ modulo $N^2$.

\begin{lem}\label{lem1.9} There exists a basis of $\cal R$ such that the following products modulo $\cal N^2$ hold true:

\begin{equation}\label{eq13}
[n_{\alpha_{i}},x_j]\equiv \alpha_{i,j}n_{\alpha_i}, \quad s+1\leq i\leq k, \quad  1\leq j \leq s.
\end{equation}
\end{lem}

\begin{proof} For $s+1\leq i\leq k $ by induction on $h$ we prove the following products:
\begin{equation}\label{eq14}
[n_{\alpha_i},x_j]\equiv \alpha_{i,j}n_{\alpha_i}, 1\leq j\leq h-1, \quad
[n_{\alpha_i},x_j]\equiv\alpha_{i,j}n_{\alpha_i}+\sum\limits_{t=i+1}^{k}
\prod\limits^{h-1}(i,t)m_{i,j}^t, \quad  h\leq j \leq s.
\end{equation}

Thanks to \eqref{eq12} the base of induction is obvious. Assume that \eqref{eq14} is true for $h-1.$ In the case of $\alpha_{i,h}\neq \alpha_{t,h},\ (i+1\le t\le k)$ we use the change of basis \eqref{eq3.4.1} assuming $A=\alpha_{i,h}-\alpha_{i+p,h}, l=h$ and $r=k-i.$ Then we derive
$[n_{\alpha_i}',x_{h}]\equiv\alpha_{i,h}n_{\alpha_i}'.$ Consequently, one can assume that

$$[n_{\alpha_i},x_{h}] \equiv\alpha_{i,h}n_{\alpha_i}+
\sum\limits_{t=i+1}^{k}\prod\limits^{h}(i,t)m_{i,h}^t, \quad  s+1\leq i\leq k.$$

Consider the following congruences with $s+1\leq i\leq k$:

$$\begin{array}{lllll}
0&\equiv & &[n_{\alpha_i},[x_j,x_{h}]]{\equiv}[[n_{\alpha_i},x_j],x_{h}]-
[[n_{\alpha_i},x_{h}],x_j]\\[3mm]
&\equiv&&\alpha_{i,j}\sum\limits_{t=i+1}^{k}
\prod\limits^{h}(i,t)m_{i,h}^t+\sum\limits_{t=i+1}^{k}
\prod\limits^{h-1}(i,t)\Big(\alpha_{t,h}m_{i,j}^t+\sum\limits_{p=t+1}^{k}
\prod\limits^{h}(t,p)\nu_{t,h}^p(m_{i,j}^t)\Big)\\[3mm]
&&-&\alpha_{i,h}\sum\limits_{t=i+1}^{k}
\prod\limits^{h-1}(i,t)m_{i,j}^t-\sum\limits_{t=i+1}^{k}
\prod\limits^{h}(i,t)\Big(\alpha_{t,j}m_{i,{h}}^t+\sum\limits_{p=t+1}^{k}
\prod\limits^{h-1}(t,p)\mu_{t,j}^p(m_{i,{h}}^t)\Big)\\[3mm]
&\equiv&&\sum\limits_{t=i+1}^{k}\Big(\prod\limits^{h}(i,t)
(\alpha_{i,j}-\alpha_{t,j})m_{i,h}^t+
\prod\limits^{h-1}(i,t)(\alpha_{t,h}-\alpha_{i,h})m_{i,j}^t\Big)\\[3mm]
&& -&\sum\limits_{t=i+1}^{k-1}\sum\limits_{p=t+1}^{k}
\prod\limits^{h-1}(i,t)\prod\limits^{h}(t,p)\nu_{t,h}^p(m_{i,j}^t)-
\sum\limits_{t=i+1}^{k-1}\sum\limits_{p=t+1}^{k}
\prod\limits^{h-1}(i,t)(t,p)\delta_{\alpha_{i,h},\alpha_{t,h}}
\mu_{t,j}^p(m_{i,{h}}^t)\\[3mm]
&\equiv& &\begin{array}{cc}\sum\limits_{r=1}^{k-i}\underbrace{\Big(
\prod\limits^{h}(i,i+r)(\alpha_{i,j}-\alpha_{i+r,j})
m_{i,h}^{i+r}+\prod\limits^{h-1}(i,i+r)(\alpha_{i+r,h}-
\alpha_{i,h})m_{i,j}^{i+r}\Big)}\\[3mm]
A(r)\\[3mm]
\end{array}\\[3mm]
&&+&\begin{array}{cc}\sum\limits_{r=1}^{k-i-1}\underbrace{
\sum\limits_{t={i+1}}^{k-r}\prod\limits^{h-1}(i,t)
\prod\limits^{h}(t,t+r)\nu_{{t},h}^{t+r}(m_{i,j}^t)} \ -\\[3mm]
B(r)\\[3mm]
\end{array}
\begin{array}{cc}
\sum\limits_{t=i+1}^{k-1}\underbrace{\sum\limits_{p=t+1}^{k}
\prod\limits^{h}(i,t)\prod\limits^{h-1}(t,p)\mu_{t,j}^p(m_{i,{h}}^t)}.\\[3mm]
C(t)\\[3mm]
\end{array}\end{array}$$

{\bf Step 1.} Comparing elements contained in $\cal N_{\alpha_{i+1}}$, we obtain $A(1)=0$. It implies
$$\begin{array}{ccccccc}\prod\limits^{h}(i,i+1)(\alpha_{i,j}-\alpha_{i+1,j})m_{i,h}^{i+1}&=&0, &
\prod\limits^{h-1}(i,i+1)(\alpha_{i+1,h}-\alpha_{i,h})m_{i,j}^{i+1}&=&0.\end{array}$$

Since $\alpha_i\neq \alpha_{i+1}$, the above equalities deduce
\begin{equation}\label{eq15}
\prod\limits^{h}(i,i+1)m_{i,h}^{i+1}=0, \quad
\prod\limits^{h-1}(i,i+1)m_{i,j}^{i+1}=
\prod\limits^{h}(i,i+1)m_{i,j}^{i+1} \ \mbox{for any} \ s+1 \leq i \leq k-1.
\end{equation}

Due to arbitrariness of $i$, equalities \eqref{eq15} can be written in general form

\begin{equation}\label{eq16}
\prod\limits^{h}(t,t+1)m_{t,h}^{t+1}=0, \quad
\prod\limits^{h-1}(t,t+1)m_{t,j}^{t+1}=
\prod\limits^{h}(t,t+1)m_{t,j}^{t+1} \ \mbox{for any} \ t\in \{i, \dots,k-1\}.
\end{equation}

According to the arbitrariness of choice $n_{\alpha_i}$ and relations \eqref{eq16} we derive $B(1)=0$.

Taking into account \eqref{eq15} and

$$C(i+1)=\prod\limits^{h}(i,i+1)[m_{i,h}^{i+1},x_j]-
\prod\limits^{h}(i,i+1)\alpha_{t,j}m_{i,h}^{i+1},$$

we conclude $C(i+1)=0$.

{\bf Step 2.} The comparison of elements that lie in $\cal N_{\alpha_{i+2}}$  leads to $A(2)=0$. Now applying the same arguments as in Step 1, we get
$B(2)=C(i+2)=0$.

Continuing in a such way, finally, we obtain $A(r)=0$ for any $r\in\{1, \dots, k-i\}$. Consequently, we get
$$\begin{array}{ccccccc}
\prod\limits^{h}(i,i+r)(\alpha_{i,j}-\alpha_{i+r,j})
m_{i,h}^{i+r}&\equiv&0, &
\prod\limits^{h-1}(i,i+r)(\alpha_{i+r,h}-\alpha_{i,h})m_{i,j}^{i+r}&\equiv&0.\end{array}$$

Since $\alpha_{i}\neq \alpha_{i+r},$ it implies

$$\begin{array}{ccccccc}\prod\limits^{h}(i,i+r)m_{i,h}^{i+r}
&\equiv&0,&\prod\limits^{h-1}(i,i+r)m_{i,j}^{i+r}
&\equiv&\prod\limits^{h}(i,i+r)m_{i,j}^{i+r}.\end{array}$$

Thus, \eqref{eq14} is proved. Assuming in \eqref{eq14} $h=s+1$, we obtain \eqref{eq13}.
\end{proof}


In the next lemma we give the exact expressions for the products $[n_{\alpha_{i}},x_j],\ 1\le i\neq j \le s$ modulo $N^2$.

\begin{lem}\label{lem7.1} There exists a basis of $\cal R$ such that the following products modulo $\cal N^2$ hold true:

\begin{equation}\label{eq13.1}
\begin{array}{lll}
[n_{\alpha_i},x_j]\equiv 0,&1\leq i\neq j\leq s.\\[1mm]
\end{array}
\end{equation}
\end{lem}

\begin{proof} In order to obtain the products \eqref{eq13.1} we shall use arguments similar as in the proof of the previous lemma. Namely, by induction on $h$ we prove the following products

\begin{equation}\label{eq19.1}
[n_{\alpha_i},x_j]\equiv 0, \quad 1\leq j\leq h-1, \quad
[n_{\alpha_i},x_j]\equiv\sum\limits_{t=s+1}^{k}\prod\limits_{q=1}
^{h-1}\delta_{\alpha_{t,q},0}m_{i,j}^t,\quad h\leq j \leq s.
\end{equation}

Thanks to \eqref{eq12} the base of induction is obvious. Assuming that \eqref{eq19.1} is true for $h-1$, in the case of $\alpha_{t,h}\neq0,\ (s+1\le p\le k)$ the application of \eqref{eq3.4.1} with $A=-\alpha_{s+p,h}, \ l=h, \ r=k-s$  leads to $[n_{\alpha_i}',x_{h}]=0.$ Consequently, one can assume
$$[n_{\alpha_i},x_h]\equiv \sum\limits_{t=s+1}^{k}\prod\limits_{q=1}
^{h}\delta_{\alpha_{t,q},0}m_{i,h}^t, \quad 1\leq i\leq s.$$

The following congruences with $1\leq i\leq s, \ h+1\leq j \leq s$ and $i\neq j$:

$$0\ \equiv\ [n_{\alpha_i},[x_j,x_{h}]]=[[n_{\alpha_i},x_j],x_{h}]-
[[n_{\alpha_i},x_{h}],x_j]\equiv
\sum\limits_{t=s+1}^{k}\prod\limits_{q=1}^{h-1}\delta_{\alpha_{t,q},0}\Big(\alpha_{t,h}m_{i,j}^t-
\delta_{\alpha_{t,h},0}\alpha_{t,j}m_{i,h}^t\Big)$$

imply

$$\prod\limits_{q=1}^{h-1}\delta_{\alpha_{t,q},0}\alpha_{t,h}m_{i,j}^t\ \equiv\ 0,\quad
\prod\limits_{q=1}^{h}\delta_{\alpha_{t,q},0}\alpha_{t,j}m_{i,h}^t\ \equiv\ 0,\quad s+1\le t\le k.$$

As condition~(iii) of Definition~\ref{defn3.2} guarantees that $\alpha_t \neq 0$, it follows that the products~\eqref{eq19.1} hold for $h$. Taking $h = s + 1$ completes the proof of the lemma. \end{proof}

\begin{lem}\label{lem1.10} There exists a basis of $\cal R$ such that the following products modulo $\cal N^2$ hold true:

\begin{equation}\label{eq18}
[n_{\alpha_{i}},x_i]\equiv \alpha_{i,i}n_{\alpha_{i}}, \quad [x_i,x_j]\equiv 0, \quad 1\leq i, j\leq s.\\[1mm]
\end{equation}

\end{lem}

\begin{proof} Let us fix an arbitrary $i\in\{1, \dots, s\}$. In the case of $\alpha_{i,i}\neq \alpha_{t,i},\ (s+1\le t\le k)$, by using \eqref{eq3.4.1} for $A=\alpha_{i,i}-\alpha_{s+p,i}, \ l=i, r=k-s,$
one can assume $[n_{\alpha_{i}},x_i]\equiv\alpha_{i,i}n_{\alpha_{i}}+\sum\limits_{t=s+1}^{k}
\delta_{\alpha_{i,i},\alpha_{t,i}}m_{i,i}^t.$

Taking into account the products \eqref{eq13}, consider the following chain of congruences:

$$\begin{array}{lll}
0&\equiv&[n_{\alpha_i},[x_i,x_j]]=[[n_{\alpha_i},x_i],x_j]-[[n_{\alpha_i},x_j],x_i]\\[1mm]
&\equiv&\sum\limits_{t=s+1}^{k}\Big((\alpha_{i,i}-
\alpha_{t,i})m_{i,j}^t+\delta_{\alpha_{i,i},\alpha_{t,i}}
m_{i,i}^t(\alpha_{t,j}-\alpha_{i,j})\Big).\end{array}$$

Hence, for any $s+1\leq t\leq k$ we obtain

$$(\alpha_{i,i}-\alpha_{t,i})m_{i,j}^t=0, \quad \delta_{\alpha_{i,i},\alpha_{t,i}}
m_{i,i}^t(\alpha_{t,j}-\alpha_{i,j})=0, \quad 1\leq i \neq j \leq s, \quad s+1\leq t \leq k.$$

Therefore, taking into account that $\alpha_i\neq \alpha_t$ we obtain $[n_{\alpha_{i}},x_i]\equiv\alpha_{i,i}n_{\alpha_{i}}.$ Hence, the first products of \eqref{eq18} are obtained.

Due to Lemma \ref{lem1.8}, we have $[x_i,x_j]\equiv\sum\limits_{t=s+1}^kv_{i,j}^t.$ Taking the change of basis as follows

$$x_1'=x_1-\sum\limits_{\tiny\begin{array}{cc}t=s+1 \\[0,1mm] \alpha_{t,2}\neq0\end{array}}^{k}
\frac{\delta_{\alpha_{t,1},0}v_{1,2}^t}{\alpha_{t,2}},\quad x_i'=x_i+\sum\limits_{\tiny\begin{array}{cc}t=s+1 \\[0,1mm] \alpha_{t,1}\neq0\end{array}}^{k}\frac{v_{1,2}^t}{\alpha_{t,1}}, \ 2\le i\le s,$$

one can assume $[x_1,x_{2}]\ \equiv\ \sum\limits_{t=s+1}^{k} \delta_{\alpha_{t,1},0}\delta_{\alpha_{t,2},0}v_{1,2}^t$ and
$[x_1,x_i]\ \equiv\ \sum\limits_{t=s+1}^{k}\delta_{\alpha_{t,1},0}v_{1,i}^t,\ 3\le i\le s.$


From Jacobi identity for the triple $\{x_1, x_2, x_i\}$ with $3\le i\le s$, we conclude
\begin{equation}\label{eq3.11}
\delta_{\alpha_{t,1},0}
\delta_{\alpha_{t,2},0}\alpha_{t,i}v_{1,2}^t-
\delta_{\alpha_{t,1},0}\alpha_{t,2}v_{1,i}^t+
\alpha_{t,1}v_{2,i}^t=0, \quad s+1\leq t \leq k.
\end{equation}

An analysis of \eqref{eq3.11} for possible cases of parameters $\alpha_{t,i}$ leads to
\begin{equation}\label{xx}[x_1,x_2]\ \equiv\ 0,\quad [x_1,x_i]\ \equiv\ \sum\limits_{t=s+1}^{k}\delta_{\alpha_{t,1},0}\delta_{\alpha_{t,2},0}v_{1,i}^t,\ 3\le i\le s.\end{equation}

By induction on $h$ with $1\le i\le s$ and  $i<j$, we prove the following products

\begin{equation}\label{eq1111}
[x_i,x_j]\equiv0,\ 2\le j\le h,  \quad [x_i,x_j]\equiv\sum\limits_{t=s+1}^{k}\prod\limits_{q=1}^{h}
\delta_{\alpha_{t,q},0}v_{i,j}^t,\ 1\le i\le h,\ h+1\le j\le s.\end{equation}

Thanks to \eqref{xx}, we have the base of induction. Assuming  that congruences \eqref{eq1111} hold for $h$, we prove they hold for $h+1.$

Setting
$$x_i'=x_i-\sum\limits_{\tiny\begin{array}{cc}t=s+1 \\[0,1mm] \alpha_{t,h}\neq0\end{array}}^{k}\prod\limits_{q=1}^{h}
\delta_{\alpha_{t,q},0}\frac{v_{i,h+1}^t}{\alpha_{t,h+1}}, \ 1\le i \le h,$$ we can assume
$$[x_i,x_{h+1}]\equiv\sum\limits_{t=s+1}^{k}\prod\limits_{q=1}^{h+1}
\delta_{\alpha_{t,q},0}v_{i,h+1}^t,\  1\le i \le h.$$

The equality with $i < h+1\le j:$
$$[x_{i},[x_{h+1},x_{j}]]=[[x_i,x_{h+1}],x_{j}]-[[x_i,x_{j}],x_{h+1}]$$
implies
$$\prod_{q=1}^{h}
\delta_{\alpha_{t,q},0}\delta_{\alpha_{t,h+1},0}\alpha_{t,j}v_{i,h+1}^t-
\prod_{q=1}^{h}
\delta_{\alpha_{t,q},0}\alpha_{t,h+1}v_{i,j}^t+
\alpha_{t,i}v_{h+1,j}^t=0, \quad s+1\leq t \leq k.$$

If $\alpha_{t,i}\neq0$ for any values of $t$, then $[x_i,x_{h+1}]\equiv 0$, i.e., \eqref{eq1111} holds true for $h+1$.

If there exists some $t$ such that $\alpha_{t,i}=0$, then
we obtain
$$\left\{\begin{array}{llllll}
\alpha_{t,j}v_{i,h+1}^t=0, & if & \alpha_{t,h+1}=0,\\[1mm]
v_{i,j}^t=0, &if &\alpha_{t,h+1}\neq0.\\[1mm]
\end{array}\right.$$

Obviously, $\alpha_{t,h+1}\neq0$ implies $[x_i,x_{h+1}]\equiv0.$ Thanks to condition (iii) of Definition \ref{defn3.2} the condition $\alpha_{t,h+1}=0$ deduces existence of $j$ such that $\alpha_{t,j}\neq0$, which implies $[x_i,x_{h+1}]\equiv0.$ Thus, \eqref{eq1111} is proved.

Finally, putting $h=s$ in \eqref{eq1111} we complete the proof of the second products of \eqref{eq18}.
\end{proof}

Now we present the main result on the description of maximal solvable extensions of a $d$-locally diagonalizable nilpotent Lie algebra.

\begin{thm} \label{mainthm} There exists a unique (up to isomorphism)  maximal solvable extension of a $d$-locally diagonalizable nilpotent Lie algebra and it is isomorphic to an algebra of the type $\cal R_{\cal T}.$
\end{thm}
\begin{proof} Let $\mathcal{R} = \mathcal{N} \oplus \mathcal{Q}$ be a maximal solvable extension of a $d$-locally diagonalizable nilpotent Lie algebra $\mathcal{N}$, and let
$\mathcal{Q} = \operatorname{Span}\{x_1, \dots, x_s\}$. Denote by $\tau$ the nilindex of $\mathcal{N}$. First, we prove by induction on $m$ the following product relations modulo $\mathcal{N}^m$ for any $m \leq \tau$:

\begin{equation}\label{eq1.10.1}
\left\{\begin{array}{llllllll}
[n_{\alpha_i},x_i] & \equiv & \alpha_{i,i}n_{\alpha_i}, & 1\leq i\leq s,\\[1mm]
[n_{\alpha_i},x_j]&\equiv & 0,& {1\leq i\neq j \leq s},\\[1mm]
[n_{\alpha_i},x_j]&\equiv & \alpha_{i,j}n_{\alpha_i},& s+1\leq i\leq k,\   1\leq j \leq s,\\[1mm]
[x_i,x_j]&\equiv & 0,  & 1\leq i, j\leq s.\\[1mm]
\end{array}\right.
\end{equation}

The base of induction with $m=2$ is true due to Lemmas \ref{lem1.9} -- \ref{lem1.10}. Let us suppose that \eqref{eq1.10.1} is true for $m$.
Then the products \eqref{eq1.10.1} modulo $\cal N^{m+1}$ will have the following forms:

$$[n_{\alpha_i},x_i]\equiv \alpha_{i,i}n_{\alpha_i}+\sum\limits_{p=s_1}^{s_m}m_{i,i}^p, \ \ [n_{\alpha_i},x_j]\equiv \sum\limits_{p=s_1}^{s_m}m_{i,j}^p, \ \
[n_{\alpha_i},x_j]\equiv \alpha_{i,j}n_{\alpha_i}+\sum\limits_{p=s_1}^{s_m}m_{i,j}^p, \ \ [x_i,x_j]\equiv \sum\limits_{p=s_1}^{s_m}v_{i,j}^p,$$
where $\cal N_{\alpha_{s_1}}\oplus \dots \oplus \cal N_{\alpha_{s_m}} \subseteq\cal N^{m}\setminus \cal N^{m+1}$.


From congruences below with $1\leq i \leq k$ and $1\leq j, t \leq s$
$$\begin{array}{lll}
0&\equiv&[n_{\alpha_i},[x_j,x_t]]=[[n_{\alpha_i},x_j],x_t]-
[[n_{\alpha_i},x_t],x_j]\\[3mm]
&\equiv&\alpha_{i,j}m_{i,t}^{s_1}-\alpha_{s_1,j}m_{i,t}^{s_1}
+\alpha_{s_1,t}m_{i,j}^{s_1}-\alpha_{i,t}m_{i,j}^{s_1}+
\sum\limits_{p=s_1+1}^{s_{m}}(*)n_{\alpha_p},\\[3mm]\end{array}$$
we derive $(\alpha_{i,j}-\alpha_{s_1,j})m_{i,t}^{s_1}=(\alpha_{i,t}-\alpha_{s_1,t})m_{i,j}^{s_1}.$ Since $\alpha_i\neq\alpha_{s_1}$, there exists $t$ such that $\alpha_{i,t}\neq\alpha_{s_1,t}.$

If $\alpha_{s_1,j}-\alpha_{i,j}=0,$ then $m_{i,j}^{s_1}=0$.

If $\alpha_{s_1,j}-\alpha_{i,j}\neq0,$ then $\frac{m_{i,j}^{s_1}}{\alpha_{s_1,j}-\alpha_{i,j}}=\frac{m_{i,t}^{s_1}}{\alpha_{s_1,t}-\alpha_{i,t}}$ and taking the change of basis \eqref{eq3.4.1} for
$A=\alpha_{i,i}-\alpha_{s_i,i}, l=i$ and $r=m$ we obtain

$$[n_{\alpha_i}, x_j]\equiv \alpha_{i,j}n_{\alpha_i}, \quad 1\leq i \leq k, \quad 1\leq j \leq s.$$

Taking into account the fact that in Jacobi identity the basis is natural, one can conclude that
$$[n_{\alpha_i}, x_j]\equiv \alpha_{i,j}n_{\alpha_i}, \quad s+1\leq i \leq n,\ $$
where $\alpha_{i,j}\in \mathbb{C}$ is belongs to the set of  roots such that
$$t_j(n_{\alpha_i})=\alpha_{i,j}n_{\alpha_i}, 1\leq i\leq n,\ 1\le j\le s.$$

Applying the same arguments as in the proof of Lemma \ref{lem1.10} modulo $\cal N^{m+1}$ one can obtain
$$[x_i,x_j]\equiv 0, \  1\leq j< i \leq s.$$

Therefore, \eqref{eq1.10.1} is proved. Setting in \eqref{eq1.10.1} $m=\tau+1,$
we conclude that pairwise commuting derivations ${ad_{x_i}}_{|\cal N}$
$(1\leq i \leq s)$ act diagonally on generators of $\cal N$. Consequently, ${ad_{x_i}}_{|\cal N}$ act diagonally on $\cal N$. To sum up, there exists a basis of $\cal N$ such that for any $n_{\alpha}\in \cal N_{\alpha}$ and $\alpha\in W$ we have

\begin{equation}\label{eq4.12.1}
[n_{\alpha},x_j]=\alpha(x_j)n_{\alpha}, \quad [x_i,x_j]=0, \quad 1\leq i, j \leq s.
\end{equation}

Since all changes of elements of the nilradical $\cal N$ are obtained through the superpositions of inner automorphisms, the multiplications table of $\cal N$ remains unchanged. Consequently, we have proved $\cal R \cong \cal R_{\cal T}$. Thanks to Remark \ref{rem3.4} we obtain the uniqueness, up to isomorphism, of maximal extension $\cal R$. \end{proof}

Note that in the case of a split $d$-locally diagonalizable nilpotent algebra $\mathcal{N}$, that is,
$\mathcal{N} = \bigoplus\limits_{i=1}^{p} \mathcal{N}_i$,
it follows from \cite{Leger1} that a maximal torus on $\mathcal{N}$ decomposes as
$\bigoplus\limits_{i=1}^p \mathcal{T}_i,$
where $\mathcal{T}_i$ is a maximal torus on  $\mathcal{N}_i$. Consequently, a maximal solvable extension of $\mathcal{N}$ is isomorphic to $\bigoplus\limits_{i=1}^{p}\mathcal{R}_{\mathcal{T}_i}.$

Example \ref{exam15} as well as Gorbatsevich's example shows  that the conditions (ii)-(iii) of  $d$-locally diagonalizability of nilradical in Theorem \ref{mainthm} are essential.

\begin{rem}
It should be noted that Theorem \ref{mainthm} gives the positive answer to \v{S}nobl's Conjecture under the assumption that $\mathfrak{n}$ is complex $d$-locally diagonalizable nilpotent Lie algebra.
\end{rem}

When condition~(i) of Definition~\ref{defn3.2} holds but condition~(ii) fails, the structure of maximal solvable extensions of a nilpotent Lie algebra $\mathcal{N}$ is no longer uniquely determined.

\begin{prop} \label{prop3.12} Let $\mathcal{N}$ be a nilpotent Lie algebra such that condition~(i) of Definition~\ref{defn3.2} is satisfied, whereas condition~(ii) is not. Then $\mathcal{N}$ admits at least two non-isomorphic maximal solvable extensions.
\end{prop}
\begin{proof} Part a). Consider a maximal solvable Lie algebra $\mathcal{R} = \mathcal{N} \oplus \mathcal{Q}$ such that, for some $x_i \in \mathcal{Q}$, the operator $\operatorname{ad}_{x_i|\mathcal{N}}$ has a non-diagonal projection and restriction on a certain root subspace with respect to a maximal torus $\cal T$ on $\mathcal{N}$. According to \cite{Leger1}, the derivation $\operatorname{ad}_{x_i|\mathcal{N}}$ admits a decomposition $\operatorname{ad}_{x_i|\mathcal{N}} = d_0 + d_1,$ where $d_0$ and $d_1$ are commuting diagonalizable and nilpotent derivations, respectively. By condition~(i), one has $d_0 \in \mathcal{T}$. Consider two maximal solvable extensions $\cal R_{\cal T}$ and $\widetilde{\cal R}=\cal N \rtimes \widetilde{\cal T},$ where $\widetilde{\cal T}=\operatorname{Span}\{\cal T, d_1\}$. These extensions are non-isomorphic, as the projection of the products
$\operatorname{ad}_{x_i}(\mathcal{N}_{\alpha})$ onto $\mathcal{N}_{\alpha}$ is diagonalizable, which contradicts the property of $\operatorname{ad}_{x_i|\mathcal{N}}$. \end{proof}

The following proposition shows that non-maximal solvable extensions of a nilpotent Lie algebra satisfying condition~(i) of Definition~\ref{defn3.2} are not unique.

\begin{prop} \label{prop3.13} A nilpotent Lie algebra $\mathcal{N}$ satisfying condition~(i) of Definition~\ref{defn3.2} admits at least two non-isomorphic, non-maximal solvable extensions.
\end{prop}
\begin{proof} Thanks to the second part of Remark~\ref{rem3.4}, a maximal solvable extension has dimension $\dim \mathcal{R} = \dim \mathcal{N} + \operatorname{rank}(\mathcal{N}).$ Let $\operatorname{rank}(\mathcal{N}) = s$, and denote by
$\mathcal{T} = \operatorname{Span}\{t_{\alpha_1}, t_{\alpha_2}, \dots, t_{\alpha_s}\}$ a maximal torus of $\mathcal{N}$.
Consider the root space decomposition $\mathcal{N} = \mathcal{N}_{\alpha_1} \oplus \mathcal{N}_{\alpha_2} \oplus \dots \oplus \mathcal{N}_{\alpha_s} \oplus \dots \oplus \mathcal{N}_{\alpha_n}$
with respect to the $\mathcal{T}$, and fix an arbitrary integer $1 \le r \le s-1$.

Set $$\mathcal{T}_1 = \operatorname{Span}\{t_{\alpha_1}, t_{\alpha_2}, \dots, t_{\alpha_r}\},
\qquad
\mathcal{T}_2 = \operatorname{Span}\{t_{\alpha_1}, \dots, t_{\alpha_r} + \cdots + t_{\alpha_s}\}.$$
One can verify that the solvable Lie algebras $\mathcal{R}_{\mathcal{T}_1}$ and $\mathcal{R}_{\mathcal{T}_2}$ both have dimension $\dim \mathcal{N} + r$ and are non-isomorphic.
Indeed, setting
$\Delta = \{\mathcal{N}_{\alpha_{r+1}}, \dots, \mathcal{N}_{\alpha_s}\} \cap \big(\mathcal{N} \setminus \mathcal{N}^2\big),$ we have $\Delta \neq \emptyset$ and
$$\mathcal{R}_{\mathcal{T}_1}^2 \cap \Delta = 0,
\qquad
\mathcal{R}_{\mathcal{T}_2}^2 = \mathcal{R}_{\mathcal{T}_1}^2 \cup \Delta,$$
which implies that $\dim \mathcal{R}_{\mathcal{T}_1}^2 < \dim \mathcal{R}_{\mathcal{T}_2}^2.$
Hence, $\mathcal{R}_{\mathcal{T}_1}$ and $\mathcal{R}_{\mathcal{T}_2}$ are non-isomorphic.
\end{proof}

\section{Further discussion and open directions}

In this section, we extend the main results of this work to the setting of Lie superalgebras. We examine the innerness of derivations in solvable Lie algebras and discus on vanishing of higher-order cohomology groups. Finally, we formulate natural conjectures and outline direction for future research.

\subsection{Maximal solvable extensions of $d$-locally diagonalizable nilpotent Lie superalgebras}
\

In this subsection, we generalize the construction of maximal solvable extensions from $d$-locally diagonalizable nilpotent Lie algebras to certain nilpotent Lie superalgebras. By focusing on even superderivations and maximal tori, we establish an analogue of Theorem~\ref{mainthm}, providing an explicit description of these extensions in the superalgebra setting.

In \cite{Wang2002}, the definition of a torus of a Lie superalgebra is given in a manner analogous to that for Lie algebras.
Since the diagonalizability of an odd superderivation implies its nullity, one may assume that a torus of a complex Lie superalgebra $\mathcal{L}$ is an abelian subalgebra of the Lie algebra $\operatorname{Der}(\mathcal{L})_0$ consisting of diagonalizable endomorphisms. According to the fact that an abelian algebra consisting of diagonalizable elements is fully reducible~\cite{Jac1}, and by Mostow's theorem on the conjugacy of maximal fully reducible subalgebras of a linear Lie algebra~\cite{Mostow}, one can conclude that a maximal torus of a nilpotent Lie superalgebra $\mathcal{N}$ is a maximal fully reducible subalgebra of the linear Lie algebra $\operatorname{Der}(\mathcal{N})_0$. Consequently, any two maximal tori of a nilpotent Lie superalgebra $\mathcal{N}$ are conjugate under an inner automorphism belonging to the radical of the derived algebra of $\operatorname{Der}(\mathcal{N})_0$. The dimension of a maximal torus of a nilpotent Lie (super)algebra is denoted by $\operatorname{rank}(\mathcal{N})$.

From now on we shall consider a solvable Lie superalgebra $\cal R=\cal R_0\oplus \cal R_1$ which satisfies the condition: $[\cal R_1,\cal R_1]\subseteq [\cal R_0,\cal R_0]$. This condition ensures fulfillment of an analogue of Lie's theorem for Lie superalgebras \cite{Kac}. Let us decompose a solvable Lie superalgebra $\mathcal{R}=\mathcal{N}\oplus \mathcal{Q}$ into direct sum of vector subspaces $\mathcal{N}$ and $\mathcal{Q}$, where $\mathcal{N}$ is its nilradical and $\mathcal{Q}$ is a complementary subspace. It is known that $\mathcal{Q}\subseteq \mathcal{R}_0$ and $ad_{x|\mathcal{N}}$ is non-nilpotent operator for any $x\in \mathcal{Q}$ (see Propositions 3.12 -- 3.13 in \cite{super}). This implies that the operator $ad_{{x}|\mathcal{N}}, x\in \mathcal{Q},$ which is an even superderivation, admits a decomposition $ad_{{x}|\mathcal{N}}=d_{0}+d_{1},$ where $d_{0}$ and $d_{1}$ are mutually commuting diagonalizable and nilpotent even superderivations of $\mathcal{N}$, respectively. Furthermore, Lie's theorem guarantees the upper-triangularity of $ad_{{x}|\mathcal{N}}$.

The concept of $d$-locally diagonalizability for Lie superalgebras is defined similarly to that for Lie algebras because of an even superderivation is just a derivation.

Thus, a maximal solvable extension of $d$-locally diagonalizable nilpotent Lie superalgebra $\cal N$ can be constructed by even superderivations of $\cal N$. Therefore, Theorem \ref{mainthm} holds true for a superalgebra case. So, the main result related to Lie superalgebras is the following:

\begin{thm} \label{mainthm2}
Let $\mathcal{N}$ be a $d$-locally diagonalizable nilpotent Lie superalgebra.
Then there exists a unique maximal solvable extension of $\mathcal{N}$, up to isomorphism, that satisfies an analogue of Lie's theorem. Moreover, this extension is isomorphic to a solvable Lie superalgebra of the form $\mathcal{R}_{\mathcal{T}}$.
\end{thm}

In view of importance of a solvable Lie (super)algebra $\cal R_{\cal T}$, below we present the construction of a maximal torus of a nilpotent Lie (super)algebra $\cal N$.

Let $\cal N=\cal N_0\oplus \cal N_1$ be a nilpotent $d$-locally diagonalizable Lie superalgebra with $\cal N_0=Span\{e_1, \dots, e_n\}$, $\cal N_1=Span\{f_1, \dots, f_m\}$ and the table of multiplications

$$\left\{\begin{array}{llll}
[e_i,e_j]=\sum\limits_{t=1}^{n}a_{i,j}^te_t,\quad 1\le i,j\le n ,\\[1mm]
[e_i,f_j]=\sum\limits_{p=1}^{m}b_{i,j}^pf_p,\quad 1\le i\le n,\ \quad 1\le j\le m \\[1mm]
[f_i,f_j]=\sum\limits_{q=1}^{n}c_{i,j}^qe_q,\quad 1\le i,j\le m.\\[1mm]
\end{array}\right.$$

For $i, j, t$ such that $a_{i,j}^t\neq 0,\ b_{i,j}^p\neq0$ and $c_{i,j}^q\neq0${\color{blue},} we consider the system of the linear equations

$$S_{e,f}: \quad \left\{\begin{array}{lllllll}
\alpha_{i}+\alpha_{j}=\alpha_{t}, \\[1mm]
\alpha_{i}+\beta_{j}=\beta_{p}, \\[1mm]
\beta_{i}+\beta_{j}=\alpha_{q},\\[1mm]
\end{array}\right.$$
in the variables $\alpha_1, \dots, \alpha_n$ as $i, j, t$ run from $1$ to $n$ and $\beta_1, \dots, \beta_m$ as $i, j, p$ run from $1$ to $m$.

Following paper \cite{Leger1} we denote by $r\{e_1, \dots, e_n, f_1,\dots, f_m\}$ the rank of the system $S_{e,f}.$ Setting $r\{\cal N\}=min \  r\{x_1, \dots,x_n, y_1,\dots , y_m\}$ as $\{x_1, \dots, x_n, y_1,\dots , y_m\}$ runs over all bases of $\cal N$. Similar to Lie algebras case (see \cite{Leger1}) for a nilpotent Lie superalgebra $\cal N$ over an algebraically closed field the equality one can establish $dim \mathcal{T}=dim \cal N -r\{\cal N\}$ holds true.

It is notably that a diagonal transformation $d=diag(\alpha_1, \dots, \alpha_n,\beta_1,\dots,\beta_m)$ is a derivation of $\cal N$ if and only if $\alpha_i,\beta_j$  are solutions of the system $S_{e,f}$.

We denote the free parameters in the solutions to the system $S_{e,f}$ by $\alpha_{i_1},\dots, \alpha_{i_s},
\beta_{j_1},\dots, \beta_{j_t}$. Making renumeration the basis elements of $\mathcal{N}$ we can assume that $\alpha_{1},\dots, \alpha_{s},\beta_1,\dots,\beta_t$ are free parameters of $S_{e,f}$. Then we get
$$\alpha_{i}+\beta_j=\sum\limits_{p=1}^{s}\lambda_{i,j}\alpha_{p}+
\sum\limits_{q=1}^{t}\mu_{i,j}\alpha_{q},\ s+1\leq i\leq n,\ t+1\le j\le m.$$

Consider the basis $\{(\alpha_{1,i},  \dots, \alpha_{n,i}, \beta_{1,j},  \dots, \beta_{m,j}) \ |  \ 1\leq i\leq s, \ 1 \le j\le t\}$ of fundamental solutions of the system $S_{e,f}$. Then the diagonal matrices $\{diag(\alpha_{1,i}, \dots, \alpha_{n,i},\beta_{1,j},  \dots, \beta_{m,j}) \ | \ 1\leq i\leq s,\ 1\le j\le t\}$ forms a basis of a maximal torus of $\cal N$. It should be noted that in the case of nilpotent Lie algebras, one constructs a maximal torus analogously to the above method by taking $\cal N_1=0$.

Thus, having $\cal N$ and its a maximal torus of $\cal T$ one can assume that a solvable Lie (super)algebra $\cal R_{\cal T}=\cal N \rtimes \cal T$ is constructed.

In general, the description (up to isomorphism) of a maximal solvable extension of a given nilpotent Lie (super)algebra $\mathcal{N}$ is carried out using the standard approach (an extension of Mubarakzjanov's method), which often involves lengthy computations.
Moreover, the solvable Lie (super)algebras of the form $\mathcal{R}_{\mathcal{T}}$ do not, in general, describe all maximal solvable extensions of $\mathcal{N}$. However, in the case of maximal solvable extensions of a $d$-locally diagonalizable nilpotent Lie (super)algebra, together with the condition that the extension satisfies an analogue of Lie's theorem (which holds automatically in the Lie algebra case), Theorems~\ref{mainthm} and~\ref{mainthm2} allow us to describe all maximal solvable extensions in a considerably simpler way, namely, through the explicit construction of the (super)algebra $\mathcal{R}_{\mathcal{T}}$.

For instance, the alternative method presented above is applicable to the Lie algebra cases considered in \cite[Proposition~6]{ancochea}, \cite[Theorem~3]{NdWi}, \cite[Theorem~3]{SnWi}, \cite[Lemma~6]{TrWi}, and \cite[Theorem~4]{WaLiDe}, as well as to the Lie superalgebra cases discussed in \cite[Theorem~4.3]{Camacho} and \cite[Theorem~4.1]{Camacho1}.

\subsection{On innerness of solvable Lie algebras}
\

In this subsection, we investigate the conditions under which maximal solvable extensions of nilpotent Lie algebras admit only inner derivations, and we provide situations that give rise to outer derivations in non-maximal extensions.

The next result establishes sufficient conditions ensuring the existence of an outer derivation in a non-maximal solvable extension of a nilpotent Lie algebra.
\begin{prop} \label{thmH1neq0}
Let $\mathcal{R} = \mathcal{N} \oplus \mathcal{Q}$ be a non-maximal solvable extension of a nilpotent Lie algebra $\mathcal{N}$, and let $\{x_1, \dots, x_s\}$ be a basis of $\mathcal{Q}$.
If one of the following conditions holds
\begin{itemize}
    \item[(i)] $\operatorname{Center}(\mathcal{R}) \neq 0;$
    \item[(ii)] the diagonal parts of the operators $\operatorname{ad}_{x_i}$, in their upper-triangular matrix representations, generate a torus of $\mathcal{R}$,
\end{itemize}
then $\mathcal{R}$ admits an outer derivation.
\end{prop}
\begin{proof}  If $\mathcal{R}$ satisfies condition~(i), then, by Proposition~\ref{prop1}, the existence of an outer derivation of $\mathcal{R}$ follows immediately.

Assume that $\mathcal{R}$ satisfies condition~(ii). By virtue of condition~(i), we may suppose that $\operatorname{Center}(\mathcal{R}) = 0$.
Since the nilradical $\mathcal{N}$ is a characteristic ideal of $\mathcal{R}$ (it is invariant under every derivation of $\mathcal{R}$), we conclude that the diagonal entries of the operators $\operatorname{ad}_{x_i}|_{\mathcal{N}}$ in their upper-triangular forms generate a torus of $\mathcal{N}$. Therefore,  $\dim \mathcal{Q} \leq \operatorname{rank}(\mathcal{N}).$ Each operator $\operatorname{ad}_{x_i}$ can be decomposed as
$\operatorname{ad}_{x_i} = d_i + d_{n_i},$
where $d_i$ and $d_{n_i}$ are commuting diagonalizable and nilpotent derivations of $\mathcal{R}$, respectively. By condition~(ii), $d_i$ and $d_{n_i}$ can be represented by diagonal and strictly upper-triangular matrices.

If $d_{n_i} \notin \operatorname{InDer}(\mathcal{R})$ for some $i$, the claim follows immediately.
Otherwise, assume $d_{n_i} \in \operatorname{InDer}(\mathcal{R})$ for all $i$, so that $d_{n_i} = \operatorname{ad}_{y_i}$ for certain elements $y_i \in \mathcal{R}$.
Define the subspace $\widetilde{\mathcal{Q}} = \operatorname{Span}\{z_i = x_i - y_i\}.$
Then $\widetilde{\mathcal{Q}}$ is complementary to $\mathcal{N}$ and forms an abelian subalgebra of $\mathcal{R}$.
Indeed, using the diagonal nature of $\operatorname{ad}_{z_i}$ and the identity
$0 = [\operatorname{ad}_{z_i}, \operatorname{ad}_{z_j}] = \operatorname{ad}_{[z_i, z_j]},$ we conclude that $[z_i, z_j] \in \operatorname{Center}(\mathcal{R}) = 0$.
Thus, $\widetilde{\mathcal{Q}}$ is abelian algebra that acts diagonally on $\mathcal{N}$.

If $\dim \widetilde{\mathcal{Q}} = \dim \mathcal{Q} < \operatorname{rank}(\mathcal{N})$, then $\operatorname{ad}(\widetilde{\mathcal{Q}})|_{\mathcal{N}}$ forms a torus of $\mathcal{N}$ properly contained in a maximal torus $\mathcal{T}$ of $\mathcal{N}$.
Choose $\widetilde{d} \in \mathcal{T} \setminus \operatorname{ad}(\widetilde{\mathcal{Q}})|_{\mathcal{N}}$ and extend it to a derivation of $\mathcal{R}$ by setting $\widetilde{d}(\widetilde{\mathcal{Q}}) = 0$.
Then $\widetilde{d}$ is an outer derivation of $\mathcal{R}$.

Now suppose $\dim \mathcal{Q} = \operatorname{rank}(\mathcal{N})$.
In this case, $\widetilde{\mathcal{Q}}$ coincides with a maximal torus $\mathcal{T}$ of $\mathcal{N}$.
Since $\mathcal{R}$ is a non-maximal extension of $\mathcal{N}$, there exists a solvable extension $\mathcal{N} \oplus \mathcal{V}$ such that $\mathcal{V} = \operatorname{Span}\{\widetilde{\mathcal{Q}}, v\}$, where $\widetilde{\mathcal{Q}}$ acts diagonally on $\mathcal{N}$.
Note that the derivations ${\operatorname{ad}_v}_{|\mathcal{N}}$ and ${\operatorname{ad}_{z_i}}_{|\mathcal{N}},$ $1 \leq i \leq s,$ are nil-independent, and the diagonal part of ${\operatorname{ad}_v}_{|\mathcal{N}}$ in its upper-triangular form is not a derivation of $\mathcal{N}$.
Hence, ${\operatorname{ad}_v}_{|\mathcal{N}}$ defines an outer derivation of $\mathcal{R}$.
\end{proof}

Recall, a centerless Lie algebra that admits only inner derivations is called {\it complete algebra.} For a given Lie algebra $\cal L$ we denote by $\cal C(\cal T)$ the commutator of $\cal T$ in $\operatorname{Der}(\cal L).$

During establishing the innerness of all derivation of maximal solvable extensions we shall use the following result.
\begin{thm} \cite{Hsie}\label{thmcomplete} Let $\cal L$ be a Lie algebra and $\cal T$ be a torus on $\cal L$ such that $0\notin W$. Then $\cal L \rtimes \cal C(\cal T)$ is equal to $Der(\cal L \rtimes \cal T)$ and it is a complete Lie algebra.
\end{thm}

\begin{rem} \label{rem2.4} It should be noted that
$\cal L \rtimes \cal C(\cal T)$ is isomorphic to the algebra $\cal L \rtimes \big(\cal C(\cal T) + \cal U\big)$ for $\cal U\subseteq \operatorname{InDer}(\cal L).$ Indeed, setting $\cal U$  and taking the basis transformation of the form $c_i'=c_i-u_i$ for appropriate basis elements $c_i$ of $\cal C(\cal T)$ and elements $u_i\in \cal U$, we obtain that $\cal L \rtimes \cal C(\cal T) \cong \cal L \rtimes \big(\cal C(\cal T) + \cal U\big).$
\end{rem}

Below we prove the main result of this section that establish the innerness of all derivations in maximal solvable extensions  of $d$-locally diagonalizable nilpotent Lie algebras.

\begin{thm} \label{thm5.4} A maximal solvable extension of a $d$-locally diagonalizable nilpotent Lie algebra admits only inner derivations.
\end{thm}
\begin{proof} Thanks to Theorem \ref{mainthm} it suffices to consider an algebra of the form $\cal R_{\cal T}.$ Taking into account Theorem \ref{thmcomplete} and centerless of the algebra $\cal R_{\cal T}$ (because of condition (iii) of Definition  \ref{defn3.2}) we obtain
$$Der(\cal N\rtimes\cal T)\cong \cal N\rtimes \cal C(\cal T) \supseteq \cal N\rtimes\cal T \cong {\rm ad}(\cal N\rtimes\cal T).$$
Therefore, to complete the proof of the theorem, it is enough to establish that  $\cal C(\cal T)=\cal T.$ It is not difficult to check that  $d\in \cal C(\cal T)$ if and only if $d(\cal N_{\alpha}) \subseteq \cal N_{\alpha}$ for any root spaces $\cal N_{\alpha}, \alpha\in W.$

Let $d$ be an arbitrary element of the space $\mathcal{C}(\mathcal{T})$, and let
$d = d_0 + d_1$ be its decomposition into commuting diagonalizable and nilpotent derivations.
It is sufficient to consider the case $d_1 \neq 0$. Indeed, if $d_1 \neq 0$, then
$\operatorname{Span}\{\mathcal{T}, d\}$ forms a torus on $\mathcal{N}$ which, by the maximality of $\mathcal{T}$, must coincide with $\mathcal{T}$. Hence, $d \in \mathcal{T}$.

Fix $\alpha \in W$ such that $d_1(\mathcal{N}_{\alpha}) \neq 0$.
If $d_0(\mathcal{N}_{\alpha}) = 0$, then by considering the set of all nil-independent derivations containing the derivation $t_{\alpha} + d_1$, we obtain a contradiction with the $d$-local diagonalizability.
If $d_0(\mathcal{N}_{\alpha}) \neq 0$, then again, by considering the set of all nil-independent derivations containing the derivation $d$, we arrive at a contradiction with the $d$-local diagonalizability.
Therefore, we conclude that $\mathcal{C}(\mathcal{T}) = \mathcal{T}$. \end{proof}

\begin{cor} \label{cor5.5}  A maximal solvable extension of a $d$-locally diagonalizable nilpotent Lie algebra is complete.
\end{cor}

\subsection{Further examples, conjectures and open direction}
\

It is well known that the quotient space of derivations by inner derivations defines the first cohomology group with coefficients in the adjoint module.
By Theorem \ref{mainthm}, we conclude that this first cohomology group vanishes.
From the cohomological point of view, a natural question arises: do the higher-order cohomology groups of maximal solvable extensions of $d$-locally diagonalizable nilpotent Lie algebras in coefficients itself also vanish?

It was shown in \cite{Ancochea222} and \cite{Leger2} that certain maximal solvable extensions of $d$-locally diagonalizable nilpotent Lie algebras possess vanishing higher-order cohomology groups.
In particular, the maximal solvable extension of an indecomposable nilpotent Lie algebra with characteristic sequence $(n_1, \dots, n_k, 1)$ and of the highest possible rank has trivial cohomology groups of order less than four, while all cohomology groups of any order for a Borel subalgebra of a semisimple Lie algebra vanish.

The next example shows that these results are not true, in general.
\begin{exam} Consider the maximal solvable extension of $d$-locally diagonalizable filiform Lie algebra given by the following non-zero products
$$\cal R: \quad
\begin{array}{lllllllll}
[e_1, e_i]= e_{i+1}, & 2 \leq i \leq n-1, & [e_2, e_i]=e_{j+2}, & 3 \leq i \leq n-2, & [e_i, x]=ie_i, & 1 \leq i \leq n.\\[1mm]
\end{array}$$

Then thanks to Theorem \ref{thm5.4} we have $H^1(\cal R, \cal R)=0$, while straightforward computations show that $H^2(\cal R, \cal R)\neq 0$.
\end{exam}

Let $\mathcal{R} = \mathcal{N} \oplus \mathcal{Q}$, where $\mathcal{Q}$ acts diagonally on $\mathcal{N}$. By the Hochschild--Serre factorization theorem, the condition $H^n(\mathcal{R}, \mathcal{R}) = 0$ implies that $H^1(\mathcal{R}, \mathcal{R}) = 0$. Consequently, solvable Lie algebras with vanishing higher cohomology should be sought among those for which all derivations are inner. In this regard, the class of maximal solvable extensions of $d$-locally diagonalizable nilpotent Lie algebras plays a particularly significant role.

The following example shows that assertions of Theorem \ref{mainthm} and Theorem \ref{thm5.4} do not true for non-Lie Leibniz algebras case.

\begin{exam} Consider the nilpotent Leibniz algebra $\cal N$ with the following multiplication table
$$\begin{array}{llll}
[e_i,e_1]=e_{i+1},& 1\le i \le n-4p-1, & [e_1,f_j]=f_{j+2p},& 1\le j\le 2p,
\end{array}$$
where $\{e_1, \dots, e_{n-p}, f_1, \dots, f_p\}$ is a basis of $\cal N$. One can verify that $\cal N$ is $d$-locally diagonalizable. In paper \cite{Adashev}, all maximal solvable extensions of $\cal N$ are described. 
Due to this description we obtain infinitely many non-isomorphic maximal solvable extensions of $\cal N$, which admit outer derivations.
\end{exam}

The methods developed for maximal solvable extensions rely on three fundamental conditions. Existing examples including nilpotent Lie algebras of maximal rank (i.e., those for which $\operatorname{rank}(\mathcal{N}) = \dim(\mathcal{N}/\mathcal{N}^2)$) and Examples~\ref{exam15}, \ref{exam3.3}, and \ref{exam3.5} demonstrate several distinct scenarios:

\begin{itemize}
\item[--] all three conditions (i)--(iii) are satisfied;
\item[--] only condition (i) holds, while (ii) and (iii) fail;
\item[--] conditions (i) and (iii) hold, but condition (ii) fails;
\item[--] only condition (iii) holds, while (i) and (ii) fail.
\end{itemize}

Since no examples are known that realize any other combination of these conditions, the following natural hypothesis arises:

\textbf{Conjecture 1.} Condition~(ii) implies the validity of conditions~(i) and~(iii). In other words, if the restrictions and projections of all nil-independent derivations on each root subspace relative to a maximal torus are diagonal, then the diagonals of all nil-independent derivations generate a maximal torus and no zero root occurs.

According to Example \ref{exam3.5}, we expect that a slight modification of the method developed in this paper will also apply to the description of maximal solvable extensions of nilpotent Lie algebras satisfying the following conditions (in upper-triangular matrix forms of all nil-independent derivations):
\begin{itemize}
\item the diagonal parts of all nil-independent derivations of a nilradical do not generate a maximal torus $\mathcal{T}$ in $\mathcal{N}$;
\item the restrictions and projections of all nil-independent derivations of a nilradical onto the blocks corresponding to multiple identical diagonal entries are scalar matrices.
\end{itemize}
The investigation devoted to this description will be the subject of future work.

The detailed computations show that the algebra in Example~\ref{exam3.5} admits, up to isomorphism, a unique maximal solvable extension. Taking this fact into account together with Conjecture~1, we can propose the following refined formulation of \v{S}nobl's conjecture:

\textbf{Conjecture 2.} A solvable extension of a nilpotent (non-characteristically nilpotent) Lie algebra is unique up to isomorphism if and only if one of the following conditions on its nilradical is satisfied (assuming that all nil-independent derivations are represented by upper-triangular matrices):
\begin{itemize}
\item[--] The restrictions and projections of all nil-independent derivations on each root subspace with respect to a maximal torus $\mathcal{T}$ are diagonal.
\item[--] The diagonals of all nil-independent derivations do not generate a maximal torus $\mathcal{T}$ on $\mathcal{N}$, and the restrictions and projections of all nil-independent derivations of the nilradical onto blocks corresponding to multiple identical diagonal entries are scalar matrices.
\end{itemize}

\

{\bf Acknowledgement} {{\small{\ We would like to thank V. Gorbatsevich for valuable discussions.}}
\\[3mm]
{\bf Data availability} {{\small{\ Data sharing not applicable to this article as no datasets were generated or analysed during
the current study.}}
\\[3mm]
{\bf Conflict of interest} {{\small{\ The authors have no competing interests to declare that are relevant to the content of this
article.}}

\end{document}